\documentclass{article}
\usepackage{fullpage}
\usepackage{amsmath}
\usepackage{amsthm}
\usepackage{amssymb}

\usepackage{amsmath}
\usepackage{amsthm}
\usepackage{url}
\usepackage{graphicx}

\newtheorem{thm}{Theorem}

\newtheorem{prop}{Proposition}
\theoremstyle{definition}

\newtheorem{remark}{Remark}

\begin{document}
%\title{}
\title{A Spectral Generalization of Von Neumann Minimax Theorem}

\author{Bahman Kalantari \\
Department of Computer Science, Rutgers University, Piscataway, NJ 08854\\
kalantari@cs.rutgers.edu}
\date{}
\maketitle

%\begin{quote} This article
%\end{quote}
\begin{abstract}
Given $n \times n$ real symmetric matrices $A_1, \dots, A_m$, the following {\it spectral minimax} property holds:
$$\min_{X \in \mathbf{\Delta}_n} \max_{y \in S_m} \sum_{i=1}^m  y_iA_i \bullet X=\max_{y \in S_m} \min_{X \in \mathbf{\Delta}_n} \sum_{i=1}^m y_iA_i \bullet X,$$
where $S_m$ is the simplex and $\mathbf{\Delta}_n$ the spectraplex. For diagonal $A_i$'s this reduces to the classic minimax.
\end{abstract}


%\newpage

%\section{Duality for Game Theory}

\section{Introdution}  Von Neumann minimax theorem in \cite{Von} is a classic result in game theory: Given  an $m \times n$ real matrix $A$,

\begin{equation} \label{Neumann0}
\min_{x \in S_n} \max_{y \in S_m} y^TAx= \max_{y \in S_m} \min_{x \in S_n} y^TAx,
\end{equation}
where $S_k=\{u \in \mathbb{R}^k: \sum_{i=1}^k u_i=1, u \geq 0\}$, the unit simplex.
Historic remarks on the proof of the theorem and its connections with linear programming duality is given in Schrijver \cite{Sch}.  In this note we prove a spectral generalization of the minimax theorem for a finite set of real symmetric matrices.  In particular, in the case of diagonal matrices the theorem reduces to the following alternate yet equivalent statement of the minimax:

\begin{thm} {\rm (Von Neumann Minimax)} Given $a_1, \dots, a_m \in \mathbb{R}^n$,
\begin{equation} \label{von1}
\min_{x \in S_n} \max_{y \in S_m} \sum_{i=1}^m  y_ia_i^Tx=\max_{y \in S_m} \min_{x \in S_n} \sum_{i=1}^m y_ia_i^Tx.
\end{equation}
\end{thm}

The minimax theorem can thus be seen as a mathematical statement on a set of $m$ vectors in $\mathbb{R}^n$.  In the next section we first give  basics on symmetric matrices, semidefinite programming and its duality. We then give statement and proof of the spectral minimax.  Here we are not concerned with any game theoretic implications of the theorem, rather the result can be viewed as a statement on a set of $m$ real $n \times n$ symmetric matrices, where the role of linear programming duality is replaced with semidefinite programming duality.

\section{Spectral Minimax Theorem} \label{sec2}
Let $\mathbb{S}^n$ denote the set of $n \times n$ real symmetric matrices. For a symmetric matrix $X$ the notations $X \succeq 0$ and $X \succ 0$ mean $X$ is positive semidefinite and  positive definite, respectively. The inner product in $\mathbb{S}^n$, also called Frobenious inner product, is denoted by any of the following equivalent notations
\begin{equation}  \label{eeq3}
\langle X, Y \rangle_F=Tr(XY)=X \bullet Y = \sum_{i=1}^n \sum_{j=1}^n x_{ij} y_{ij}.
\end{equation}

The {\it primal} semidefinite programming problem refers to the following optimization
\begin{equation} \label{primal}
\inf \big \{C \bullet X: A_i \bullet X = b_i, i=1, \dots, m, X \succeq 0 \big\},
\end{equation}
where $C, A_i \in \mathbb{S}^n$, $b_i \in \mathbb{R}$. The {\it dual} of (\ref{primal}) is
\begin{equation} \label{dual}
\sup \big \{\sum_{i=}^m b_i y_i:  \sum_{i=1}^m y_iA_i + S=C, S \succeq 0 \big \}.
\end{equation}
It is easy to show that given any feasible solution $X$ of (\ref{primal}) and any feasible solution $(y, S)$ of (\ref{dual}), $C \bullet X \geq \sum_{i=1}^m b_i y_i$. Furthermore, it is well known that in semidefinite programming, as a  conic linear programming, if there exists a feasible $X \succ 0$ and feasible $(y,S)$ with $S \succ 0$, the optimal objective value of  both problems are attained and equal (see \cite{NN}).  We shall make use of this property.

The spectral analogue of the unit simplex, called {\it spectraplex} (see \cite{Parrilo}) is the set
\begin{equation}  \label{eeq5}
\mathbf{\Delta}_n= \big \{X \in \mathbb{S}^n:  Tr(X)=1,  X \succeq 0  \big \}.
\end{equation}

\begin{prop}  \label{props1} Given $A \in \mathbb{S}^n$, if $\lambda_{\rm min}(A)$ is its minimum eigenvalue, then
\begin{equation} \label{lem1eq1}
\min \big \{A \bullet X :  X \in \mathbf{\Delta}_n \big \}=\lambda_{\rm min}(A).
\end{equation}
\end{prop}
\begin{proof} Consider the spectral decomposition of $A$, $U\Lambda U^T$, where $\Lambda={\rm diag}(\lambda)$ is the diagonal matrix of eigenvalues and  $U=[u_1, \dots, u_n]$ the corresponding matrix of eigenvectors.
Given $X \in \mathbf{\Delta}_n$, let $Y=U^T X U$. Then $Y \in \mathbb{S}^n$. Also we have
\begin{equation}
Tr(Y)=Tr(X UU^T)=Tr(X), \quad Tr (A X) = Tr(U\Lambda U^T X )= Tr(\Lambda U^T X U) = Tr(\Lambda Y).
\end{equation}
In particular, $\{Y=U^T X U: X \in \mathbf{\Delta}_n\}=\mathbf{\Delta}_n$. From these and also observing that the minimum of $\Lambda \bullet Y$ over $\mathbf{\Delta}_n$ is $\lambda_{\min}(\Lambda)$ we get,
\begin{equation}
\min \big \{A \bullet X :  X \in \mathbf{\Delta}_n \big \}=
\min \big \{ \Lambda \bullet Y :  Y \in \mathbf{\Delta}_n \big \}= \lambda_{\min}(\Lambda)= \lambda_{\min}(A).
\end{equation}
\end{proof}

\begin{prop} \label{props2} Denoting the $n \times n$ identity matrix by $I_n$,
given any $A \in \mathbb{S}^n$, we have
\begin{equation}
\lambda_{\rm min}(A) =\max \big \{t:  A- tI_n \succeq 0 \big \}.
\end{equation}
\end{prop}
\begin{proof} It is easy to show $A- t I_n$ is not PSD if and only if $t > \lambda_{\rm min}(A)$.
\end{proof}

\begin{thm} {\rm (Spectral Minimax Theorem)} Given $A_1, \dots, A_m \in \mathbb{S}^n$,
\begin{equation} \label{main}
\min_{X \in \mathbf{\Delta}_n} \max_{y \in S_m} \sum_{i=1}^m  y_iA_i \bullet X=
\max_{y \in S_m} \min_{X \in \mathbf{\Delta}_n} \sum_{i=1}^m y_iA_i \bullet X.
\end{equation}
\end{thm}

\begin{proof} For each fixed $X \in \mathbf{\Delta}_n$ it easy to see the LHS of (\ref{main}) is
\begin{equation}
\max \big \{\sum_{i=1}^m y_iA_i \bullet X: y \in S_m   \big \}= \max \big \{ A_i \bullet X: i=1, \dots, m\big \}.
\end{equation}
Thus the LHS of (\ref{main}) is equivalent to the following semidefinite programming whose infimum, by compactness of $\mathbf{\Delta}_n$, is attained
\begin{equation} \label{LGvon}
\delta_*=\min \big \{\delta : A_i \bullet X \leq \delta, i=1,\dots, m, X \in \mathbf{\Delta_n}\big \}.
\end{equation}
On the other hand, for each fixed $y \in S_m$  the RHS of (\ref{main}) is
\begin{equation} \label{eqs9}
\min \big \{(\sum_{i=1}^m y_iA_i) \bullet X: X \in \mathbf{\Delta}_n\big \}.
\end{equation}
Thus by Propositions \ref{props1} and \ref{props2} the RHS of (\ref{main}) is equivalent to
\begin{equation} \label{RGvon}
t_*=\max \big \{t:  \sum_{i=1}^m y_iA_i - tI_n \succeq 0, y \in S_m \big \}.
\end{equation}
We prove  (\ref{LGvon}) and (\ref{RGvon}) are primal-dual pair and $\delta_*=t_*$, hence
proving (\ref{main}).  Let us assume $\delta_* \geq 0$.   The case with $\delta_* <0$ can be handled analogously and will be omitted. Introducing slacks,  (\ref{LGvon}) can be written as
\begin{equation} \label{LGvonp}
\delta_*=\min \big \{\delta : A_i \bullet X + s_i -\delta =0, i=1,\dots, m, X \in \mathbf{\Delta}_n, s_i \geq 0, \delta \geq 0\big \}.
\end{equation}
We will rewrite (\ref{LGvonp}) in the primal form (\ref{primal}). In doing so, let $E^k_i$ denote the $k \times k$ matrix with $1$ as its $i$-th diagonal entry and all other entries zero. Now (\ref{LGvonp}) can be written as
\begin{equation} \label{LGvon2}
\delta_*=\min \big \{C' \bullet X':  A'_i \bullet X'=0,  i=1, \dots, m, E' \bullet X'=1, X' \succeq 0\big \},
\end{equation}
where all matrices lie in $\mathbb{S}^{n'}$, $n'=n+m+1$, and are defined as follows
\begin{equation} \label{matrices}
A'_i={\rm diag}(A_i, E^m_i,-1), \quad E'={\rm diag}(I_n, 0), \quad C'= E^{n'}_{n'}.
\end{equation}
From (\ref{dual}) the dual of (\ref{LGvon2}) is
\begin{equation} \label{RGvon2}
\max \big \{t: \sum_{i=1}^m y_i A'_i + t E' +S=C', S \succeq 0\big \}.
\end{equation}
We show  (\ref{RGvon2}) is equivalent to (\ref{RGvon}).  From the first set of $n$ linear equations in (\ref{RGvon2})  we get
\begin{equation} \label{eqno2}
\sum_{i=1}^m y_i A_i +t I_n +S^{(n)} =0,
\end{equation}
where $S^{(n)}$ denotes the top left $n \times n$ submatrix of $S$.
Since $S$ is positive semidefinite, so is $S^{(n)}$. Thus from (\ref{eqno2}) 
we may write
\begin{equation} \label{eqno2a}
\sum_{i=1}^m y_i A_i +t I_n \preceq 0.
\end{equation}
From the next $m$ equations in (\ref{RGvon2}) and the fact that diagonal entries of $S$ are nonnegative we get
\begin{equation} \label{eqno2b}
y_i \leq 0,  \quad i=1, \dots, m.
\end{equation}
Finally, consider the last equation in (\ref{RGvon2}).
Since the $(n',n')$ entry of $C'$ is $1$ and the corresponding entry of $S$, by virtue of its positive definiteness, is nonnegative,
the last  equation in (\ref{RGvon2}) gives
\begin{equation}  \label{y}
-\sum_{i=1}^m y_i \leq 1.
\end{equation}
Now changing  $y$ to $-y$, the three constraints (\ref{eqno2a}), (\ref{eqno2b}), (\ref{y}) give the following equivalent formulation of (\ref{RGvon2})
\begin{equation} \label{const}
\max \{t:  \sum_{i=1}^m y_i A_i - t I_n \succeq 0, \quad \sum_{i=1}^m y_i \leq 1, \quad y \geq 0\}.
\end{equation}
Next we argue that the constraint $\sum_{i=1}^m y_i \leq 1$ in (\ref{const}) can be replaced with $\sum_{i=1}^m y_i =1$. Otherwise, given  any optimal solution  by scaling $y$ we can increase the maximum value of $t$ subject to the constraints in (\ref{const}).  What remains to be verified is that both primal and dual problems have interior points.

For any positive definite $X \in \mathbf{\Delta}_n$ we can choose $\delta >0$ such that $A_i \bullet X < \delta$.  Thus in (\ref{LGvonp})
the slack   $s_i >0$, for all $i=1, \dots, n$.  This implies (\ref{LGvon2}) has an interior point $X'={\rm diag}(X,s_1, \dots, s_m, \delta)$.
Next we show  (\ref{RGvon2}) has a feasible solution ($y,t,S)$ with $S$ strictly positive definite.  To prove this we set $y_i =1/m$, $i=1, \dots, m$ and pick $t$ so that $\sum_{i=1}^m y_i A_i -t I_n$ is negative definite and define
\begin{equation}
S={\rm diag} \big (-\sum_{i=1}^m y_iA_i + t I_n, \frac{1}{m}I_m, 2 \big).
\end{equation}
$S$ is positive definite and from the definition of $A_i'$, $E'$ and $C'$ in
(\ref{matrices}) it follows that $(y,t,S)$ is a feasible solution  to (\ref{RGvon2}). Thus $\delta_*=t_*$ and the proof is complete.
\end{proof}

\begin{remark}  By replacing minimizations in the propositions with maximization we obtain analogous results based on which the following spectral maximin property can be proven, interchanging min and max over the simplex and spectraplex (an interchange whose proof in the standard maximin is superfluous),

$$\max_{X \in \mathbf{\Delta}_n} \min_{y \in S_m} \sum_{i=1}^m  y_iA_i \bullet X=\min_{y \in S_m} \max_{X \in \mathbf{\Delta}_n} \sum_{i=1}^m y_iA_i \bullet X.$$
\end{remark}

%\bigskip

\end{document}